\documentclass[12pt,reqno,letterpaper]{CAA}
\usepackage{mathrsfs}
\usepackage{color}

\large\normalsize

\newtheorem{theorem}{Theorem}[section]
\newtheorem{lemma}[theorem]{Lemma}

\newtheorem{proposition}[theorem]{Proposition}
\newtheorem{definition}[theorem]{Definition}
\newtheorem{example}[theorem]{Example}
\newtheorem{remark}[theorem]{Remark}
\numberwithin{equation}{section}
\newcommand{\R}{\mathbb{R}}

\newcommand{\N}{\mathbb{N}}

\def\eps{\varepsilon}

\def\om{\Omega}
\def\barom{\overline{\Omega}}
\def\Dom{\partial \Omega}

\def\W{\mbox{W}}

\def\Lip{{\rm Lip}}

\newcommand{\Pb}{\mathscr{P}}
\newcommand{\C}{\mathcal{C}}
\newcommand\w{\rightharpoonup}
\newcommand\ws{\stackrel{*}{\rightharpoonup}}
\newcommand{\MP}{\mathcal{P}}

\newcommand{\M}{\mathcal{M}}

\newcommand{\dsp}{\displaystyle}
\newcommand{\clomega}{\overline{\Omega}}
\newcommand{\Wa}{\mathcal{W}_1}

\newcommand{\CAAnewsection}[1]{\addtocounter{section}{1}\section*{#1}\setcounter{equation}{0}\setcounter{theorem}{0}}
\newcommand{\CAAnewsubsection}[1]{\subsection*{#1}\noindent}

\begin{document}
\noindent {{\it Communications in Applied Analysis} {\bf xx} (200x),
  no. N, xxx--xxx}
               
\title[The $\infty$-eigenvalue problem and optimal transportation\hfill]{\mbox{}\\[1cm]The
  $\infty$-eigenvalue problem and \\ a problem of optimal transportation}
\author[\hfill T. Champion]{Thierry Champion$^1$}
\author[L. De Pascale]{Luigi De Pascale$^2$}
\author[C. Jimenez]{Chlo\'e Jimenez$^3$}
\begin{abstract}
\vspace{-.3cm}

\begin{center}
$^1$ Institut de Math\'ematiques de Toulon et du Var (Imath),\\
Universit\'e du Sud Toulon-Var, Avenue de l'Universit\'e, BP 20132,\\
83957 La Garde cedex, FRANCE \\
\& Centro de Modelamiento Matem\'atico (CMM), Universidad de Chile,\\
  Blanco Encalada 2120, Piso 7, Santiago de Chile, CHILE\\
{\em E-mail:} champion\symbol{64}univ-tln.fr\\[0.3cm]

$^2$ Dipartimento di Matematica Applicata, Universit\'a di Pisa\\
 Via Buonarroti 1/c, 56127  Pisa, ITALY \\
{\em E-mail:} depascal\symbol{64}dm.unipi.it \\[0.3cm] 

$^3$Laboratoire de Math\'ematiques de Brest, UMR 6205\\
Universit\'e de Brest, 6 avenue le Gorgeu,\\
CS 93837, F-29238 BREST Cedex 3 FRANCE \\
{\em E-mail:} chloe.jimenez\symbol{64}univ-brest.fr
\end{center}
\vspace{-.2cm}

\end{abstract}

\thanks{\hspace{-.5cm}
\rm{Received October 13, 2008	 \hfill 1083-2564 \$15.00 \copyright Dynamic Publishers, Inc.}}

\maketitle
\thispagestyle{empty}

\centerline{\it Dedicated to Prof. Espedito De Pascale}
\centerline{\it in occasion of his retirement.}

\vskip\baselineskip
{\footnotesize\noindent{\bf ABSTRACT.}
The so-called eigenvalues and eigenfunctions of the infinite Laplacian
$\Delta_\infty$ are defined through an asymptotic study of that
of the usual $p$-Laplacian $\Delta_p$, this brings to a characterization
via a non-linear eigenvalue problem for a PDE satisfied in the viscosity sense.
In this paper, we obtain an other characterization of the first eigenvalue
via a problem of optimal transportation, and recover properties of the
first eigenvalue and corresponding positive eigenfunctions.

\noindent{\bf AMS (MOS) Subject Classification.} 99Z00. Insert subject classifications}

\CAAnewsection{1. INTRODUCTION}
 
An eigenvalue of the $p$-Laplacian is a real number $\lambda \in \R$
such that the problem
\begin{equation*}
\left\{\begin{array}{ll}
-div (|Du|^{p-2}Du)= \lambda |u|^{p-2}u & \mbox{in} \ \Omega, \\
u=0 & \mbox{on}\ \partial \Omega,
\end{array}\right.
\end{equation*}
has at least one non trivial solution in $W^{1,p}_0(\Omega)$.
Here solution is intended in the distributional sense and $\Omega$ is
assumed to be a regular, bounded, open subset of $\R^N$. 
Much is unknown about the eigenvalues of the $p$-Laplacian and we will
give a short presentation of some related open questions in section  \S 2.

In this paper, we shall focus on the asymptotic of the above
eigenvalue problem as the parameter $p$ goes to $+\infty$.
This is a standard strategy in analysis (for example in the homogenization and
relaxation theories) to look at the asymptotic problem and then to try to deduce
qualitative and quantitative informations on the approximating problems
and the limit problem as well as reasonable conjectures.

The asymptotic as $p \to \infty$ of the $p$-Laplacian eigenvalue
problem  was
introduced in \cite{JuuLinMan} and then perfectioned in
\cite{JuuLinMan2, JuuLin, ChaDep3}. 
In these papers the authors proved that if $(\lambda_p)_{N<p<\infty} $ is a
generalized sequence of eigenvalues of the $p$-Laplacian such that 
$\lim_{p \to \infty}\lambda_p ^{1/p}=\Lambda$ and $u_p$ are
corresponding eigenfunctions such that $\|u_p\|_p \leq C$ and $u_p \to
u$ uniformly, then $u$ is a viscosity solution of 
\begin{equation}\label{inftyeigen}
\left\{ \begin{array}{ll}
\min\{|\nabla u| - \Lambda u \, , \, -\Delta_\infty u \}=0 & \mbox{in}
\,\,\, \{u>0\},\\
-\Delta_\infty u =0 & \mbox{in} \,\,\, \{u=0\},\\
\max\{-|\nabla u| - \Lambda  u \, , \, -\Delta_\infty u \}=0 &
\mbox{in} \,\,\, \{u<0\},\\
\end{array} \right.
\end{equation}
where the infinite Laplacian of $u$ is given by
$\displaystyle \Delta_\infty u = \sum_{i,j} u_{x_i \,x_j}u_{x_i}u_{x_j}$.
According to the definition given in \cite{JuuLin} this means that $u$ is
an eigenfunction of the $\infty$-Laplacian for the $\infty-$eigenvalue $\Lambda$.

The aim of this paper is to introduce a different asymptotic problem
as $p \to \infty$ of the first eigenvalue problem  which relates the problem to an optimal
transportation problem, to start an analysis of the limiting problem as well as
propose some related questions and a
few answers. The idea that a transport equation appears in the limit as
$p \to \infty$ goes back to \cite{BatDibMan}. The explicit connection
of this limit with the optimal transportation problem was first
exploited in \cite{EvGa} and in the setting of the eigenvalues
problems appeared also in \cite{GarManPerRos}.

The main reason to focus our study on the first eigenvalue is that
the restriction $u_{\lambda,V}$ of an eigenfunction $u_\lambda$
(for some eigenvalue $\lambda$ of the $p$-Laplacian operator) to one of its
nodal domains $V$ is indeed an eigenfunction for the first eigenvalue of
the corresponding $p$-Laplacian operator for this domain $V$.
A close study on the first eigenvalue (and related eigenfunctions)
of the $p$-Laplacian operator is then of great help to understand
the properties of the eigenfunctions of higher eigenvalues.
This was in particular illustrated in \cite{JuuLin}.

The paper is organized as follows. Section \S 2 
is devoted to review basic notions and results concerning the eigenvalues
of the $p$-Laplacian. In section \S 3 
we propose a new asymptotic analysis as $p$ goes to $\infty$, and make the
link with an optimal transport problem in section \S 4. 
In the final section \S 5 
we show how the proposed asymptotic analysis may be applied to obtain
some informations on the limits obtained. 


\CAAnewsection{2. DEFINITIONS AND PRELIMINARY RESULTS}\label{notations}


\CAAnewsubsection{Nonlinear eigenvalues of the $p$-Laplacian}\label{nonlindefi}\hfill

\noindent We shall denote by $\|\cdot\|_p$ the usual norm of $L^p(\om)$ (or
$L^p(\om;\R^N)$ when dealing with the gradient of some element of $W_0^{1,p}(\om)$).

An eigenvalue of the $p-$Laplacian operator $-\Delta_p$ is a real number $\lambda$
for which the problem
\begin{equation}\label{Lp}
(\mbox{P}_p^\lambda)\quad \quad \quad \left\{\begin{array}{ll}
-\Delta_p u := -div (|\nabla u|^{p-2}\nabla u) = \lambda |u|^{p-2}u & \mbox{in} \ \Omega, \\
u=0 & \mbox{on}\ \partial \Omega,
\end{array}\right.
\end{equation}
has a non-zero solution in $W^{1,p}_0(\Omega)$.
This problem (and its generalizations to monotone elliptic operators) 
has been widely studied in the literature and for more
detailed treatment we refer to \cite{AppDepVig,Bro,Cof,GarPer1,GarPer2,JuuLin,Lin}.
Much is still unknown about the eigenvalues of the $p-$Laplacian
operator. A good understanding of the set of the eigenvalues would
permit some progress on more general nonlinear equations involving the
$p$-Laplacian (e.g. a good definition of jumping nonlinearity) as well
as some progress on parabolic equations involving the $p$-Laplacian. Let us report
some classical results. It is known
that $\lambda$ is an eigenvalue if and only
if it is a critical value for the Rayleigh quotient 
$$ v \mapsto \frac{\int_\Omega |\nabla v|^pdx}{\int_\Omega |v|^pdx}
\,\,\, \left(= \frac{\| \nabla v \|_p^p}{\| v \|_p^p} \right)$$
which is a Gateaux differentiable functional on $W^{1,p}_0(\Omega)$ outside the origin. 
Moreover, a sequence $(\lambda_p^k)_{k \geq 1}$ of eigenvalues can be obtained as
follows (we refer to \cite{GarPer1} and \cite{Lin} for details).
Denote by $\Sigma_p^k(\om) $ the set of those subsets $G$ of $W^{1,p}_0 (\Omega)$
which are symmetric (i.e. $G=-G$), contained in the set $\{ v \,:\, \|v\|_p=1\}$, strongly compact
in $W^{1,p}_0(\Omega)$ and with Krasnoselskii genus $\gamma(G)\geq k$
(we refer to \cite{struwe} for more details on the Krasnoselskii genus), and set
\begin{equation*}
\lambda _p ^k \,\,=\,\, \inf_{G \in \Sigma_k^p(\om)} \, \, \sup _{u \in G}
\|\nabla u\|_p ^p.
\end{equation*}
Then each $\lambda _p ^k$ defined as above is an eigenvalue of the $p$-Laplacian operator
and $\lambda _p ^k \to +\infty$ as $k \to \infty$.
Moreover $\lambda_p^1$ is the smallest eigenvalue of $-\Delta_p$,
it is simple (see \cite{BelKaw1} for a short proof) and the operator
$-\Delta_p$ doesn't have any eigenvalue between $\lambda_p^1$ and $\lambda_p^2$.

A second sequence $(\mu_p^k)_k$ of eigenvalues was introduced in Theorem 5
of  \cite{DraRob}. This sequence is also obtained by a $\inf -\sup$
operation but in this case the inf operation is performed on a smaller class
of sets than $\Sigma^k_p$ (we refer the reader to \cite{DraRob} for more details).
It is only known that $\lambda_p^1=\mu_p^1$ and $\lambda_p^2
=\mu_p^2$.  Some interesting questions related to our analysis are the following:
does it hold  $\lambda_p^k=\mu_p^k$ for all $p$ and $k$?
Is it true that $\{\lambda _p^k \}_{k \geq 1}$ is the entire set of eigenvalues?

The relevance of these questions may be also understood in the light of
a theorem of Fredholm alternative for the $p$-Laplacian which appear
in \cite{AppDepVig} (namely theorem 12.12 therein).

Finally let us report a basic estimate for the first eigenvalue which
is a consequence of the following characterization:
\begin{equation}\label{min} 
\lambda_p^1 = \min \left\{\int_\Omega |\nabla u|^p dx \ | \ u \in
W_0^{1,p}(\Omega), \ \|u\|_p=1 \right\}.
\end{equation}
Denote by 
$$ R_1= \sup \{r | \ \exists x_0 \ \mbox{s.t.} \ B(x_0,r) \subset
\Omega\},$$
the radius of the biggest ball inscribed in $\Omega$ then 
 \begin{lemma} For each $p \in [1, \infty),$  we have
   $(\lambda_p^1)^{1/p} \leq \frac{1}{R_1}$ and then 
$$ \limsup_{p \to \infty} (\lambda_p^1)^{1/p} \leq \frac{1}{R_1}.$$
\end{lemma}
\begin{proof} Let $B(\overline{x}, R_1)$ be a ball inscribed in
  $\Omega$, then $v(x):= \max \{R_1-|x-\overline{x}|, 0\}$ belongs to
  $ W_0^{1,p}(\Omega)$ and it is enough to test the minimality in
  (\ref{min}) against $v / \|v\|_p $ to obtain the desired estimate.
\end{proof}

As the main focus of the paper will be on the generalized sequence of
the first eigenvalue we will simplify the notations and write $\lambda_p$
for $\lambda_p^1$.
Up to subsequences we may then assume that $(\lambda _p)^{1/p} \to \Lambda_\infty$
and we will in fact prove that $\Lambda_\infty =\frac{1}{R_1}.$ This has already been
proved in \cite{JuuLinMan} and then in \cite{JuuLin,ChaDep3}. Here we
deduce this equality from a minimality property of $u_p$ and from the
Monge-Kantorovich (or optimal transportation) problem 
obtained in the limit as $p \to \infty$.


\CAAnewsubsection{$\Gamma$-convergence}\hfill

\noindent A crucial tool in the analysis of this paper will be the following  concept of
$\Gamma$-convergence.
  
Let $X$ be a metric space, 
a sequence of functionals $F_n:X\to \overline \R$ is said to
$\Gamma$-converge to $F_\infty$ at $x$ if
\begin{equation}
F_\infty(x) \, = \, \Gamma-\liminf F_n(x) \, = \, \Gamma-\limsup F_n(x),
\end{equation}
where
\begin{equation}
\left\{
\begin{array}{ll}
\vspace{1mm}                               %
\Gamma-\liminf F_n(x) & \!\!\! = \, \inf \, \big \{\liminf F_n(x_n): \ \ x_n \to x \,\, in \,\, X \big \}, \\
\Gamma-\limsup F_n(x) & \!\!\! = \, \inf \, \big \{\limsup F_n(x_n): \ \ x_n \to x \,\, in \,\, X \big \}.
\end{array}\right.
\end{equation}
The $\Gamma-$convergence was introduced in \cite{DeGio}, 
for an introduction to this theory we refer to
\cite{dalmaso} and \cite{attouch}.  
We report a classical theorem which includes some
properties of the $\Gamma$-convergence
that we shall use in the following.
\begin{theorem}\label{propriegamma}
Assume that the sequence $(F_n)_{n \in \N}$ of functionals
$\Gamma$-converges to $F_\infty$ on $X$. 
Assume in addition that the sequence $(F_n)_n$ 
is equi-coercive on $X$. Then
$$
\displaystyle \lim_{n \to +\infty} \left( \inf_{x \in X} F_n(x) \right)
\, = \, \inf_{x \in X} F_\infty(x)
$$ and one has
$
\displaystyle F_\infty(x_\infty) = \inf_{x \in X} F_\infty(x)
$
for any cluster point $x_\infty$ of a sequence $(x_n)_{n \in \N}$
such that
$$
\forall n \in \N \quad \quad F_n(x_n) \leq \inf_{x \in X} F_n(x) + \varepsilon _n
$$
with $\varepsilon_n\to 0$ as $n \to \infty$.
\end{theorem}

\CAAnewsection{3. THE ASYMPTOTIC BEHAVIOR AS $p \to \infty$.}\label{asymptsect}

Recall that, for any $p>N$, $\lambda_p$ stands for the first eigenvalue of the $p$-Laplace
operator. We shall denote by $u_p$ the unique corresponding eigenfunction which is positive
in $\om$ and such that
\begin{equation}\label{norm1}
\|u_p\|_p \,\,=\,\, \left(\int_\om u_p^p(x) dx \right)^{1/p} \,\, = \,\, 1.
\end{equation}
We also introduce the following measures:
\begin{equation}\label{defsigp}
\sigma_p:= \frac{|\nabla u_p|^{p-2}\nabla u_p}{\lambda_p}dx, \quad
f_p:=u_p^{p-1}dx,\quad
\mu_p:= \frac{|\nabla u_p|^{p-2}}{\lambda_p}dx.
\end{equation}

\begin{lemma}\label{lemEST} The above measures satisfy the following inequalities for $p>2$:
\begin{equation*}
\begin{array}{c}
\displaystyle{\int_{\Omega}\left\vert{\nabla u_p\over \lambda_p^{1/p}}\right\vert^p\ dx=1,\quad
\int_{\Omega}d\vert f_p\vert \le \vert \Omega \vert^{1/p},}\\
\displaystyle{\int_{\Omega}d \vert \mu_p\vert\ \le\vert\Omega\vert^{2/p},\quad
\int_{\Omega}d\vert \sigma_p\vert\le\vert\Omega\vert^{1/p}.}
\end{array}
\end{equation*}
Then there exists $u_{\infty}\in {\rm Lip}(\Omega) \cap \C_0(\om)$ with $\|u_\infty\|_\infty =1$, 
$f_\infty\in \M_b^+(\overline{\Omega})$ a probability measure,
$\mu_\infty\in \M_b^+(\overline{\Omega})$ and $\xi_\infty\in L^1_{\mu_\infty}(\Omega)^d$ such that,
up to subsequences:
\begin{equation*}
\begin{array}{c}
\displaystyle
u_p\to u_{\infty}\mbox{ uniformly on $\barom$,}\quad
f_p\ws f_\infty\ \mbox{in $\M_b(\overline{\Omega})$},\\
\displaystyle{\mu_p\ws\mu_\infty \mbox{ in
    $\M_b^+(\overline{\Omega})$},\quad
  \sigma_p\ws\sigma_{\infty}:=\xi_\infty\mu_\infty \mbox{ in
    $\M_b(\overline{\Omega},\R^N)$}.}
\end{array}
\end{equation*}
\end{lemma}
\begin{proof}
The second bound is an easy consequence of H\"older's inequality and of the assumption
$\int\vert u_p\vert^{p}\ dx=1$.
To obtain the remaining estimates, it is sufficient to show the first equality
and then apply H\"older's inequality.
As $u_p$ solves (\ref{Lp}), by multiplying the PDE (\ref{Lp}) by $u_p$ and integrating
by parts we get
$$
\int_\om \vert \nabla u_p\vert^p\ dx= \lambda_p \int_\om \vert u_p\vert^{p}\ dx=\lambda_p.
$$

By the above estimates, for any $N\le q<+\infty$, $(u_p)_{p> q}$ is
bounded in $W_0^{1,q}(\Omega)$, more precisely, using Holder's inequality, we get:
$$\int_{\Omega} \vert \nabla u_p(x)\vert^q\ dx
\le (\int_{\Omega} \vert \nabla u_p\vert^p dx)^{q\over p} \vert \Omega \vert^{1-q/p}
=(\lambda_p^{1/p})^q\vert\Omega\vert^{1-q/p}.$$
As a consequence, fixing $q>N$, we obtain that $(u_p)_{p>q}$ is precompact in ${\mathcal C}(\barom)$
and, up to subsequences, the uniform convergence to some $u_{\infty}$ holds.

Using again the estimates above,  we get (up to subsequences) the existence
of a weak* limit $f_\infty$ for $(f_p)_p$, $\sigma_\infty$ for
$(\sigma_p)_p$ and $\mu_\infty$ for $(\mu_p)_p$ in $\M_b(\overline{\Omega})$.
Note that, as we are
on a compact set, the convergence of $(f_p)_p$ is tight.
From this convergence  it comes that $|f_\infty(\clomega)| \leq 1$. To obtain the
reverse inequality we observe that for all $p$ one has $\int u_p d f_p=1$ so that in
the limit $\int u_\infty d f_\infty=1$. On the other hand it follows from the Holder
inequality applied with $1<q<p$ that
$$
\|u_p\|_q \, \leq \, \|u_p\|_p \, |\Omega |^{\frac{1}{q}-\frac{1}{p}}
\, = \, |\Omega |^{\frac{1}{q}-\frac{1}{p}}.
$$
Taking the limit as $p \to + \infty$ and then as  $q \to + \infty$ yields $\|u_p\|_\infty \leq 1$.
Therefore
$$
1 \, = \, \int u_\infty d f_\infty
\, \leq \, \|u_\infty\|_\infty \, |f_\infty(\clomega)| \, \leq \, 1
$$
so that $f_\infty$ is a probability measure on $\clomega$.
Moreover, thanks to lemma 3.1 of
\cite{BBDP}, we can write $\sigma_{\infty}=\xi_\infty\mu_\infty$ for some
$\xi_\infty\in L^1_{\mu_\infty}(\Omega)^d$.\\
\end{proof}
We devote the rest of the paper to the properties of the limits
$u_\infty, \ f_\infty,\ \sigma_\infty,\ \mu_\infty$.

\CAAnewsubsection{A first $\Gamma$-convergence approach}\hfill

\noindent If we consider $f_p$ as known, we may introduce the following variational problem:
$$
(\MP_p)\quad \quad \quad \min_{u\in W_0^{1,p}(\Omega)}
\left\{ {1\over p\lambda_p}\int_{\Omega} \vert \nabla u(x)\vert^p\ dx-\langle f_p,u\rangle\right\}.
$$
By the definitions of $u_p$ and $f_p$, it follows that $u_p$ is the unique minimizer
of $(\MP_p)$. Moreover, since the solution set of the problem $(\mbox{P}_p^{\lambda_p})$
is spanned by $u_p$, we may consider $(\MP_p)$ as a variational formulation of (\ref{Lp}) for
$\lambda=\lambda_p$.
Then we have:
\begin{proposition}\label{thm-ptoinfty}
The sequence $(\min (\MP_p))_p$ converges to the minimum of the following optimization problem:
$$
(\MP_\infty)\quad \min\{-<f_\infty,u>:\ u\in \Lip(\Omega),\ \vert \nabla u\vert\le \Lambda_{\infty}\ \mbox{a.e.},\
u=0\ \mbox{on }\partial\Omega\},
$$
and $u_{\infty}$ minimizes $(\MP_\infty)$.
\end{proposition}
\begin{proof}
For $p>N$ let $F_p:\C_0(\om) \to \R \cup \{+\infty\}$ defined by
\begin{equation*}
F_p(u):=\left\{
\begin{array}{ll}
 \dsp{\frac{1}{p}\int _\Omega \left| \frac{\nabla u}{\lambda_p^{1/p}}\right|^pdx -
\langle f_p, u \rangle }&\ \mbox{ if } u\in W_0^{1,p}(\om),\\
+\infty, &\ \mbox{otherwise.}
\end{array}
\right.
\end{equation*}

We claim that the family $(F_p)_{p>N}$ $\Gamma$-converges in $\C_0(\om)$ to $F_\infty$ given by
$$
F_\infty(u):=\left\{
\begin{array}{ll}
-\langle f_\infty, u \rangle &\ \mbox{ if } u\in \Lip(\om)\ \mbox{and}\
\vert \nabla u\vert\le \Lambda_\infty\
\mbox{a.e. in $\om$}\\
+\infty, &\ \mbox{otherwise,}
\end{array}
\right.
$$
with respect to the norm of the uniform convergence.
We first show the $\Gamma-\liminf$ inequality, that is:
\begin{equation}\label{gliminf}
\Gamma-\liminf_{p \to +\infty} F_p \geq F_\infty
\end{equation}
Let $(v_p)_{p>N}$ converging uniformly to $v$, then we have:
\begin{equation}\label{easyPart}
\langle f_p,v_p\rangle\to \langle f_\infty,v\rangle.
\end{equation}
We shall prove that $\displaystyle \liminf_{p \to +\infty} F_p(v_p) \geq F_\infty(v)$.
We may assume that   $\displaystyle \liminf_{p \to +\infty} F_p(v_p)<+\infty$, that
is (thanks to (\ref{easyPart})):
$$
M \,\, := \,\,
\liminf_{p\to +\infty} \left(\frac{1}{p}\int_\Omega \left| \frac{\nabla v_p}{\lambda_p^{1/p}}\right|^pdx\right)<+\infty.
$$
It then remains to check that $v$ is Lipschitz continuous and satisfies
$\vert \nabla v\vert \leq \Lambda_\infty$ a.e. in $\om$.
Let $N<q<p$, then the $W^{1,q}$-norm of $({v_p\over \lambda_p^{1/p}})_p$ is bounded.
Indeed, as for $t>0$ the function
$s \mapsto {(t^s -1)\over s}$ is monotone increasing on $\, ]0,+\infty[\,$:
$$
\frac{1}{q}\int _\Omega \left| \frac{\nabla v_p}{\lambda_p^{1/p}}\right|^q dx\le
\frac{1}{p}\int _\Omega \left| \frac{\nabla v_p}{\lambda_p^{1/p}}\right|^pdx +(1/q-1/p)\vert \Omega\vert.
$$
Then, possibly extracting a subsequence we may assume
${v_p\over \lambda_p^{1/p}}\w {v\over \Lambda_\infty}$ in $W^{1,q}_0(\Omega)$ and
then:
$$
\left(\int _\Omega \left| \frac{\nabla v}{\Lambda_\infty}\right|^q dx\right)^{1/q}\le
\liminf_{p\to \infty} \left(\int _\Omega \left| \frac{\nabla v_p}{\lambda_p^{1/p}}\right|^q dx\right)^{1/q}
 \leq (q \,M - |\Omega|)^{1/q}.
$$
Letting $q$ go to $+\infty$ we get $|\nabla v |\leq \Lambda_\infty$ almost everywhere on $\om$.
This concludes the proof of (\ref{gliminf}). The $\Gamma-\limsup$ inequality,
i.e. $\Gamma-\limsup_{p \to +\infty} F_p(v) \leq F(v)$, follows by considering the constant sequence
$(v_p)_{p \geq 1} := (v)_{p \geq 1}$.

The Proposition now follows as a consequence of Theorem \ref{propriegamma} and 
of the uniform convergence of $(u_p)_p$ to $u_\infty$.
\end{proof}

We shall now see that the measure $\sigma_\infty$ plays its role in the classical
dual problem associated to $(\MP_\infty)$, as shown in Proposition \ref{min-sigma-inf}
below. We first identify the dual problem for $(\MP_\infty)$.

\begin{proposition}[Duality for the limit problem]\label{duality}
By convex duality we have:
\begin{multline}
\min(\MP_\infty)= -\min(\MP_\infty^*)
:=\\-\min_{\lambda\in{\mathcal P}(\partial \Omega)}
\min_{\sigma\in\M_b(\R^N)^N}\{\Lambda_{\infty}\int_{\R^N} \vert \sigma\vert:\
-{\rm div}(\sigma)=f_\infty-\lambda\ \mbox{in }\R^N \}.
\end{multline}
Moreover the minimum of $(\MP_\infty^*)$ can also be expressed as:
\begin{multline*}\min(\MP_\infty^*)
:= \min_{\sigma\in\M_b(\R^N)^N}
\{\Lambda_{\infty}\int_{\overline \Omega} \vert \sigma\vert:\  {\rm spt}(\sigma)\subset \overline{\Omega},\ 
-{\rm div}(\sigma)\in \M_b(\R^N)\ \mbox{and} \\ 
-{\rm div}(\sigma)=f_\infty\ \mbox{in }\Omega \}.
\end{multline*}

\end{proposition}

The equalities $-{\rm div}(\sigma)=f_\infty-\lambda\ \mbox{in }\R^N$
and $-{\rm div}(\sigma)=f_\infty\ \mbox{in }\Omega$ should be understood
in the sense of distributions, that is:
$$
-{\rm div}(\sigma)=f_\infty-\lambda\ \mbox{in }\R^N\quad\mbox{ means: }
\int \nabla\varphi\cdot \sigma=\int \varphi \,d(f_\infty-\lambda)
\quad \forall\varphi\in{\mathcal C}_c^\infty(\R^N),
$$
$$
-{\rm div}(\sigma)=f_\infty\ \mbox{in }\Omega\quad\mbox{ means: }
\int \nabla\varphi\cdot \sigma=\int \varphi\, df_\infty
\quad \forall\varphi\in{\mathcal C}_c^\infty(\Omega).
$$

The proof of Proposition \ref{duality} requires the following Lemma:
\begin{lemma}\label{regularize}
Let $u\in \Lip(\om)$ such that $\vert \nabla u\vert\le \Lambda_\infty$
a.e. in $\om$ and $u=0$ on $\partial \Omega$. Then there exists a sequence $(u_n)_n$ in
${\mathcal C}^\infty_c(\R^N)$ such that
for any $n\in\N$:
$$
u_n\rightarrow u\mbox{ uniformly in }{\overline \Omega}
$$
$$
u_n\mbox{ is }\Lambda_\infty\mbox{-Lipschitz and }u_n=0
\mbox{ on a neighborhood of }\partial \Omega.
$$
\end{lemma}

\begin{proof}
We denote by ${\tilde u}$ the function $u$ extended by $0$ outside $\Omega$. For any $\eps>0$ we set:
$$\theta_\eps(t)=\left\{
\begin{array}{l l}
 0 &\mbox{ if }\vert t\vert\le \Lambda_\infty \eps\\
t-\mbox{sign}(t)\Lambda_\infty\eps &\mbox{ if }\vert t\vert\ge \Lambda_\infty \eps.\\
\end{array}
\right.
$$
The function $\theta_\eps\circ {\tilde u}$ remains $\Lambda_\infty$-Lipschitz and  satisfies:
\begin{equation}\label{groszero}
\theta_\eps\circ {\tilde u}(x)=0\mbox{ as soon as }d(x,\partial \Omega)\le \eps.
\end{equation}
We now make a standard regularization by convolution setting for any $\eps>0$ and $n\in \N$:
$$\psi_{n,\eps}(x)=\int_{B(0,1/n)} \rho_n(x)(\theta_\eps\circ {\tilde u})(x-y)\ dy$$
where $\rho_n:={1\over n}\rho(n \times\cdot)$ is a standard mollifier
obtained from a function $\rho$ satisfying
$$
\rho \in {\mathcal C}^\infty(\R^N, [0,+\infty[),\quad {\rm spt}(\rho)\subset B(0,1),
\quad \int_{B(0,1)} \rho(x)\ dx=1.
$$
For any $n\ge {2\over \eps}$, the function $\psi_{n,\eps}$ is 
${\mathcal C}^1$, $\Lambda_\infty$-Lipschitz and, by (\ref{groszero}), equals $0$ on
$\R^N \setminus \{ x \in \Omega, d(x,\partial \Omega)\le \frac{\eps}{2} \}$.
Moreover we have the following convergences:
$$
\psi_{n,\eps}\rightarrow \theta_\eps \circ {\tilde u}\mbox{ uniformly on }{\overline \Omega}
\mbox{ as }n\to +\infty,
$$
$$\theta_\eps\circ {\tilde u} \rightarrow {\tilde u} \mbox{ uniformly on }{\overline \Omega} \mbox{ as }\eps\to 0.$$
By extracting a diagonal subsequence of $(\psi_{n,\eps})_{n,\eps}$, we get the desired sequence $(u_n)_n$.
\end{proof}

\begin{proof}[Proof of Proposition \ref{duality}]
The above lemma allows us to rewrite problem $(\MP_\infty)$ in the following way:
$$\min(\MP_\infty)=\inf\{-<f_\infty,u>:\ u\in {\mathcal C}^1(\R^N)\cap {\mathcal C}_c(\R^N),\ \vert \nabla u\vert\le \Lambda_\infty,\ u=0\mbox{ on }\partial\Omega\}.$$
We introduce the operator $A: {\mathcal C}_c(\R^N)\rightarrow {\mathcal C}_c(\R^N)^N$ of domain ${\mathcal C}^1(\R^N)\cap {\mathcal C}_c(\R^N)$
defined as $Au:=\nabla u$ for all $u$ in its domain. We also introduce the characteristic functions
$\chi_{B_{\Lambda_\infty}}$ and $\chi_C$ defined by:
$$
\chi_{B_{\Lambda_\infty}}(\Phi)=\left\{ \begin{array}{l l }
0&\mbox{ if }\  \vert\Phi(x)\vert\le \Lambda_\infty,\quad \forall x\in\R^N\\
+\infty &\mbox{ elsewhere. }
\end{array}
\right.
\leqno  \forall \Phi\in {\mathcal C}_c(\R^N)^N,
$$
$$
\chi_{C}(\varphi)=\left\{ \begin{array}{l l }
0&\mbox{ if }\ \varphi(x)=0,\quad \forall x\in \partial \Omega\\
+\infty &\mbox{ elsewhere. }
\end{array}
\right.
\leqno  \forall \varphi\in {\mathcal C}_c(\R^N). \quad
$$
We have:
\begin{eqnarray*}
\min(\MP_\infty)&=&-\max\{<f_\infty,u>-(\chi_{B_{\Lambda_\infty}}\circ A+\chi_{C})(u):\ u\in {\mathcal C}_c(\R^N)\}\\
&=&-(\chi_{B_{\Lambda_\infty}}\circ A+\chi_{C})^*(f_\infty)=-\left((\chi_{B_{\Lambda_\infty}}
\circ A)^*\bigtriangledown \chi_{C}^*\right)^{**}(f_\infty)\\
\end{eqnarray*}
where $\bigtriangledown$ is the inf-convolution, that is for all $f\in {\mathcal M}^+_b(\R^N)$:
$$
(\chi_{B_{\Lambda_\infty}}\circ A)^*\bigtriangledown \chi_{C}^*(f)=
\inf_{\lambda\in {\mathcal M}^+_b(\R^N)} \{(\chi_{B_{\Lambda_\infty}}\circ A)^*(f -\lambda)+\chi_{C}^*(\lambda)\}.
$$
Now, by classical computations, we have that for all $\lambda \in {\mathcal M}^+_b(\R^N)$
\begin{eqnarray*}
(\chi_{B_{\Lambda_\infty}}\circ A)^*(f-\lambda)
&=&\inf_{\sigma \in \mathrm{dom } A^*}
\{\chi^*_{B_{\Lambda_\infty}}(\sigma):\ A^*(\sigma)=f-\lambda\}\\
&=&\inf_{\sigma\in\M_b(\R^N)^N}\{\Lambda_{\infty}\int \vert \sigma\vert:\
-{\rm div}(\sigma)=f-\lambda\mbox{ in }\R^N\}
\end{eqnarray*}
and:
$$\chi_{C}^*(\lambda)=\sup_{u\in {\mathcal C}_c(\R^N),\ u=0\ \mbox{on }\partial\Omega}<\lambda,u>=\left\{
\begin{array}{c c}
0 &\mbox{ if } {\rm spt}(\lambda)\subset \partial \Omega\\
+\infty&\mbox{ elsewhere. }
\end{array}
\right.
$$
The inf-convolution thus gives:
$$
(\chi_{B_{\Lambda_\infty}}\circ A)^*\bigtriangledown \chi_{C}^*(f)=\inf_{\lambda\in{\mathcal M}_b^+(\partial\Omega)}
\inf_{\sigma\in\M_b(\R^N)^N}\{\Lambda_{\infty}\int d\vert \sigma\vert:\
-{\rm div}(\sigma)=f-\lambda\mbox{ in }\R^N\}$$
which happens to be a convex, lower semi-continuous function in $f$.
By consequence:
$$\min(\MP_\infty)= -\inf_{\lambda\in{\mathcal M}_b^+(\partial \Omega)}\inf_{\sigma\in\M_b(\R^N)^N}\{\Lambda_{\infty}\int d\vert \sigma\vert:\
-{\rm div}(\sigma)=f_\infty-\lambda\mbox{ in }\R^N\}.$$
We notice that if $\lambda$ is not a probability then the second infimum is $+\infty$,
otherwise it is a minimum.
This proves the thesis.
\end{proof}

The previous result of course holds for the approximating problems:

\begin{proposition}[Duality for the approximating problems]
For every $p> 1$, setting $p'={p\over p-1}$, by standard duality we have:
\begin{multline}\label{dualAp}
\min(\MP_p)= -\min(\MP_p^*)
:=-\min_{\sigma\in L^{p'}(\R^N)}\{ {1\over p'}\lambda_p^{p'-1}\int_{\overline \Omega} \vert \sigma\vert^{p'}\ dx:\ {\rm spt}(\sigma)\subset{\overline \Omega},\\  -{\rm div}\sigma \in \M_b (\R^N)\ \mbox{and} \
-{\rm div}\sigma=f_p\mbox{ in }\Omega \}.
\end{multline}
\end{proposition}

\begin{proof}[Sketch of the proof.]
As in the proof of Proposition \ref{duality}, it can be proved that:
\begin{multline*}
\min(\MP_p)
=\inf\left\{ (G\circ A+\chi_C)(u)-<f_p,u>:\ u\in {\mathcal C}^1(\R^N)\cap {\mathcal C}_c(\R^N)\right\}=\\ -(G\circ A+\chi_C)^*(f_p)
\end{multline*}
where $G(\Phi)$ is defined for all $\Phi\in {\mathcal C}_c(\R^N,\R^N)$ by $G(\Phi)={1\over p\lambda_p}\int \vert \Phi(x)\vert^p\ dx$.
Its Fenchel transform is for any $\rho\in {\mathcal M}_b(\R^N, \R^N)$:
$$
G^*(\sigma)=\left\{
\begin{array}{l l}
{1\over p'}\lambda_p^{p'-1}\int \vert \underline{\rho}\vert^{p'}\ dx\
&\mbox{ if $\rho \ll dx$ with $\rho=\underline{\rho}\ dx$,} \\
+\infty&\mbox{ otherwise}.\\   
\end{array}
\right.
$$
The rest of the proof follows that of Proposition \ref{duality}.
\end{proof}

It can now be checked that also the dual problems converge that is:
$$\min(\MP_p^*)\to\min(\MP_\infty^*).$$

More precisely, one has the following:

\begin{proposition}\label{min-sigma-inf}
The function $\sigma_p$ defined in (\ref{defsigp}) is the unique minimizer of $(\MP_p^*)$. Moreover, its limit $\sigma_\infty$
given by Theorem \ref{thm-ptoinfty} is a solution of $(\MP_\infty^*)$. In other words, setting $\lambda_\infty:=f_\infty+{\rm div}\sigma_\infty$,
the couple $(\lambda_\infty,\sigma_\infty)\in{\mathcal P}(\partial\Omega)\times{\mathcal M}_b(\R^N)^N$ minimizes $(\MP_\infty^*)$.
\end{proposition}

\begin{proof}
As $u_p$ is an eigenfunction of the p-Laplacian, recalling (\ref{defsigp}),  $\sigma_p$ is admissible for $(\MP_p^*)$.
 Moreover by Lemma \ref{lemEST}, we have:
$$\min(\MP_p)={1\over p\lambda_p}\int_\Omega \vert \nabla u_p\vert^p\ dx -<f_p,u_p>={1\over p}-\int_{\Omega} u_p^p(x)\ dx=-{1\over p'},$$
$$\mbox{and} \quad \quad {1\over p'}\lambda_p^{p'-1}\int_{\Omega} \vert \sigma_p\vert^{p'}\ dx={1\over p'\lambda_p}\int_\Omega \vert \nabla u_p\vert^p\ dx={1\over p'}.$$
Then by (\ref{dualAp}), $\sigma_p$ is a solution of $(\MP_p^*)$, the uniqueness follows from the
strict convexity of the functional $\sigma\mapsto \int \vert \sigma\vert^{p'}\ dx$.\\
Passing to the limit in the constraint of $(\MP_p^*)$, we obtain that the measure $\sigma_\infty$ satisfies $-{\rm div}(\sigma_\infty)=f_\infty$ in $\Omega$.
It then remains to prove that
$$
\min({\MP_\infty^*})\ge \Lambda_\infty \int_{\overline\Omega} \vert \sigma_\infty \vert.
$$
Following  the proof of Theorem 4.2 in \cite{BBDP},
we use the inequality ${s^{p'}\over p'}\ge s-{1\over p}$ for any $s>0$, and get:
$$\min(\MP_p^*)={1\over p'}\lambda_p^{p'-1}\int \vert \sigma_p\vert^{p'}\ dx
\ge \lambda_p^{p'-1} \left(\int \vert \sigma_p\vert\ dx -{|\Omega|\over p} \right).$$
Then, passing to the limit, by Corollary \ref{thm-ptoinfty},
we obtain:
$$
\min(\MP_\infty^*) \,\ge\, \liminf_{p\to +\infty}\, \lambda_p^{p'-1}\int \vert \sigma_p\vert\ dx
\,=\, \liminf_{p\to +\infty}\, (\lambda_p^{1/p})^{p'}\int \vert \sigma_p\vert\ dx 
\,\ge\, \Lambda_\infty\int \vert\sigma_\infty \vert.$$
\end{proof}

\CAAnewsubsection{A second $\Gamma$-convergence approach}\hfill

\noindent An other way of obtaining the problem $(\MP_\infty)$ in a limit process,
which we shall use is the following of the paper, is to define for any $p \in \,]N,+\infty]$
the functional $G_p: \M (\Omega) \times \C_0 (\Omega) \to \overline{\R}$ by
\begin{equation}
G_p(g,v)= \left\{\begin{array}{ll}
- \langle g, v \rangle  & \mbox{if} \ g \in L^{p'}, \ \|g\|_{p'} \leq
1 \ \mbox{and} \ v \in W^{1,p}_0 (\Omega), \ \|\nabla v\|_{p} \leq \lambda_p^{1/p},\\
+\infty & \mbox{otherwise.}
\end{array} \right.
\end{equation}
and
\begin{equation}
G_\infty(g,v)= \left\{\begin{array}{ll}
- \langle g, v \rangle  & \mbox{if} \ \int_\Omega d|g| \leq
1 \ \mbox{and} \ v \in W^{1,\infty}_0 (\Omega), \ \|\nabla v\|_{\infty} \leq
\Lambda_\infty,\\
+\infty & \mbox{otherwise.}
\end{array} \right.
\end{equation}

For $p\in \,]N,+\infty[\,$ it happens that the couple $(f_p,u_p)$ is a minimizer
of the functional $G_p$. Indeed by the definitions above and (\ref{min}) it comes
$$
-G_p(g,v)=\langle g, v \rangle \leq \ \|g\|_{p'}  \|v\|_{p} \leq
\frac{1 }{ \lambda_p^{1/p}}\|\nabla v\|_{p} \leq 1 = \langle f_p,u_p\rangle = -G_p(f_p,u_p).
$$

We now notice that this property does also hold in the limit $p=+\infty$:

\begin{proposition}\label{maximize}
Let $\alpha >0$, then the generalized sequence $(G_p)_{N+\alpha<p}$ is
  equicoercive and  $\Gamma$-converges to
  $G_\infty$ with respect to the $(w^* \times \rm{uniform})$-convergence.
In particular the couple $(f_\infty,u_\infty)$ is a minimizer of
the functional $G_\infty$.
\end{proposition}

\begin{proof}
We only prove the $\Gamma$-convergence, and first show the $\Gamma-\liminf$ inequality,
that is:
\begin{equation}\label{G-liminf}
\Gamma-\liminf_{p\to +\infty} G_p\ge G_\infty.
\end{equation}
Let $(g_p,v_p)\in {\rm L}^{p'}(\Omega)\times W_0^{1,p}(\Omega)$ and $(g,v)\in \M (\Omega) \times \C_0
(\Omega)$ such that $(g_p,v_p)$ converges to $(g,v)$ for the $(w^* \times \rm{uniform})$-topology. We easily have:
$$
- \langle g_p, v_p \rangle=-\int v_p\ dg_p\to -\int v\ dg=- \langle g, v \rangle;\
\int_{\Omega}d\vert g\vert= \lim_{p\to +\infty}  \|g\|_{p'} \leq 1.
$$
Moreover, for any $\varphi\in {\mathcal C}_c^\infty(\Omega)$, it holds:
$$
\left\vert\int v_p(x) \nabla\varphi(x)\ dx\right\vert
\le \|\nabla v_p\|_{p} \|\varphi\|_{p'}\le \lambda_p^{1/p}\|\varphi\|_{p'}.
$$
Passing to the limit as $p$ tends to $\infty$ this yields:
$$
\left\vert\int v(x) \nabla\varphi(x)\ dx\right\vert\le  \Lambda_\infty\|\varphi\|_{1},
$$
that is $v\in W_0^{1,\infty}(\Omega)$ and $\ \|\nabla v\|_{\infty} \leq
\Lambda_\infty$. This ends the proof of (\ref{G-liminf}).\\
Let us now prove the $\Gamma-\limsup$ inequality. Take $(g,v)\in \M (\Omega) \times  W^{1,\infty}_0 (\Omega)$
such that:
$$
\int_\Omega d|g| \leq 1, \quad \|\nabla v\|_{\infty} \leq \Lambda_\infty.
$$
By setting $v_p={\lambda_p^{1/p}\over \Lambda_\infty}\,v$, we get a sequence such that:
$$
v_p\to v\mbox{ uniformly },\quad v_p\in W_0^{1,p}(\Omega),\quad
{\| \nabla v_p\|_p\over \lambda_p^{1/p}}= {\| \nabla v\|_p\over \Lambda_\infty}\le 1.
$$
To build a sequence of measures $g_p \in  L^{p'}(\Omega)$ satisfying $\|g\|_{p'} \leq
1$, we make a regularization by convolution:
$$\forall x \in \R^N, \quad \quad \quad g_p(x):=\int \rho_p(x-y)\ dg(y) $$
where $\rho_p:={1\over p}\rho(p \times \cdot)$ is a standard mollifier
obtained as in the proof of Lemma \ref{regularize}.
We thus get a family $(g_p)_{p>N}$ in ${\mathcal C}_c^\infty(\R^N)$
such that:
$$
g_p\ws g\mbox{ in $\M (\Omega)$ and } \|g_p\|_{p'}\le \int d\vert g\vert \le 1.$$
Finally, from the properties of $(v_p)_p$ and $(g_p)_p$, we have:
$$\lim_{p\to +\infty} G_p(g_p,v_p)=G(g,v).$$
\end{proof}

\CAAnewsection{4. THE LINK WITH AN OPTIMAL TRANSPORT PROBLEM.} \label{mongesect}

A reader familiar with the Monge-Kantorovich or optimal transportation
problem already recognized in problems $(\MP_\infty)$ and
$(\MP_\infty^*)$ two of its dual formulations. Let us introduce this
connection shortly. One of the advantages in exploiting this connection is
that sometime it is possible to compute explicitly or numerically the value of the
Wasserstein distance introduced below.
For example, we will use this explicit computability in section $\S 5$ to prove that 
$\Lambda_\infty=1/R_1$.

Given two probability measures $\alpha$ and $\nu$ on $\overline \Omega
$ the Monge problem (with the Euclidean norm as cost) is the following minimization problem:
\begin{equation}\label{monge}
\inf \left\{ \int_{\overline \Omega} |x-T(x)|d\alpha \,:\, T_\sharp \alpha=\nu \right\}
\end{equation}
where the symbol $T_\sharp \mu$ denotes the push forward of $\alpha$
through $T$ (i.e. $T_\sharp \alpha(B):= \alpha (T^{-1}(B))$ for every Borel
set $B$). A Borel map $T$ such that  $T_\sharp \alpha=\nu$ is called a
transport of $\alpha$ to $\nu$ and it is called an optimal transport if
it minimizes (\ref{monge}).
It may happens that the set of transports of $\alpha$ to $\nu$ is empty
(e.g. $\alpha=\delta_0$ and $\nu=\frac{1}{2}(\delta_1 + \delta_{-1})$ or
that the minimum is not achieved (e.g. $\alpha={\mathcal H}^1
_{\{0\}\times [0,1]}, \ \  \nu=\frac{1}{2} ({\mathcal H}^1 _{\{-1\}\times [0,1]}+{\mathcal H}^1 _{\{1\}\times [0,1]}$). 
To deal with these situations in  the '40 Kantorovich proposed 
the following relaxation of the problem above 
\begin{equation}\label{kantorovich}
\min \left\{ \int_{\overline \Omega \times \clomega} |x-y|d\gamma :
\pi^1_\sharp \gamma= \alpha, \ \pi^2_\sharp \gamma=  \nu \right\}.
\end{equation}
A measure $\gamma$ such that $\pi^1_\sharp \gamma= \alpha, \
  \pi^2_\sharp \gamma=  \nu $ is called a transport plan of $\alpha$
  to $\nu$.
Notice that by the direct method of the Calculus of Variations
the minimum in (\ref{kantorovich}) is
achieved. The minimal value is usually called Wasserstein distance of
$\alpha$ and $\nu$ and it is denoted by $\Wa(\nu, \alpha)$.

Let $f_\infty \in \Pb(\Omega)$ be the measure defined in Lemma \ref{thm-ptoinfty},
and consider its Wasserstein distance from $\Pb(\partial \Omega)$,
i.e. the following  variational problem defined on $ \Pb(\partial \Omega)$
\begin{equation}\label{minform}
\inf_{\nu \in  \Pb(\partial \Omega) } \Wa(f_\infty, \nu).
\end{equation}
By definition, solving problem (\ref{minform}) is equivalent to solve
\begin{equation}\label{kantorovich2}
\inf \left\{ \int_{\clomega \times \clomega} |x-y|d\gamma :
\pi^1_\sharp \gamma= f_\infty, \ \pi^2_\sharp \gamma\in \Pb(\partial \Omega) \right\}
\end{equation}

The following proposition is a variant of the classical
  Kantorovich duality (see for example theorem 1.3 of \cite{Vil}) and
  it will help us to connect problems (\ref{kantorovich2})
with problems $(\MP_\infty)$ and $(\MP_\infty^*)$.

\begin{proposition} The following equalities hold
 \begin{equation}\label{duality1} 
\inf \left\{ \int_{\clomega \times \clomega} |x-y|d\gamma :
\pi^1_\sharp \gamma= f_\infty, \ \pi^2_\sharp \gamma\in \Pb(\partial \Omega) \right\}
\,=\,-{1\over \Lambda_\infty} \min (\MP_\infty)\,=\,{1 \over \Lambda_\infty} \min(\MP_\infty^*).
\end{equation}
\end{proposition}

An other way of expressing the link between the limit quantities obtained in
Lemma \ref{thm-ptoinfty} and the optimal transportation theory is via the following
Theorem \ref{main3}, which is the main result of this section and expresses in a
useful way the primal-dual optimality conditions coming from
Proposition \ref{duality}.

\begin{theorem}\label{main3}
The limits $(u_{\infty},f_\infty,\sigma_\infty,\xi_{\infty},\mu_\infty)$ obtained in Lemma \ref{lemEST} satisfy:
\begin{equation}\label{syst}
\left\{
\begin{array}{r c l}
\sigma_\infty&=&\xi_\infty\,\mu_\infty\\
\xi_{\infty}&=&\Lambda_{\infty}^{-1}\nabla_{\mu_\infty}u_{\infty},\quad \mu_\infty-\mbox{a.e. in }{\overline\Omega},\\
-{\rm div}(\nabla_{\mu_\infty}u_{\infty}\,.\,\mu_\infty)&=& \Lambda_{\infty}\,f_\infty,\quad\mbox{in
 the sense of distributions in } \Omega,\\
\vert \nabla_{\mu_\infty}u_{\infty}\vert&=&\Lambda_{\infty},\quad\mu_\infty-\mbox{a.e. in }{\overline\Omega}.\\
\end{array}
\right.
\end{equation}
In the above result $\nabla_{\mu_\infty}u_{\infty}$ denotes the tangential gradient
of $u_\infty$ to the measure $\mu_\infty$ (see Definition
\ref{tangrad} for details)
\end{theorem}

The proof of Theorem \ref{main3} requires to perform an integration by parts
with respect to a measure.
In order to do that
we introduce, shortly,  the notion of tangent space to a measure and of tangential gradient to a measure.
This notion has first been introduced by Bouchitt\'e, Buttazzo and
Seppecher in \cite{BBS}, the case of interest here is
developed in \cite{J}: we now recall the main points tools in our setting.

Let us define the set
 \begin{equation}\label{def-N}
 \begin{array}{l l}
 {\mathcal N}:=\bigg\{&\xi\in {\rm L}^{\infty}_{\mu_\infty}(\R^{N},\R^{N}):\ \exists (u_n)_n,\ u_n\in {\mathcal C}^1(\R^{N}),
\\&
\hskip 1cm  u_n
\rightarrow 0
 \mbox{ uniformly on }\R^{N},\ \nabla u_n\ws \xi\ \mbox{in }\sigma\left({\rm L}^\infty_{\mu_\infty},{\rm L}^1_{\mu_\infty}\right)\bigg\}
 \end{array}
 \end{equation}
where $\sigma\left({\rm L}^\infty_{\mu_\infty},{\rm L}^1_{\mu_\infty}\right)$ denotes the weak star topology of
${\rm L}^\infty_{\mu_\infty}(\R^N,\R^N)$.
We notice that when $\mu_\infty$ is not absolutely continuous with respect
to the Lebesgue measure, this set is not necessarily reduced to zero.

The following results and notions may be found in \cite{J}:

\begin{proposition}
There exists a multi-function $T_{\mu_\infty}$ from $\R^N$ to $\R^N$ such that:
$$\eta \in {\mathcal N}^\perp \Leftrightarrow \eta(x)\in T_{\mu_\infty}(x)\ {\mu_\infty}-\mbox{a.e.}x.$$
\end{proposition}

\begin{definition}
For ${\mu_\infty}-\mbox{a.e. } x$, we call $T_{\mu_\infty}(x)$
the tangent space to ${\mu_\infty}$ at $x$ and
denote by $P_{\mu_\infty}(x,\cdot)$ the orthogonal projection on $T_{\mu_\infty}(x)$.
\end{definition}

\begin{proposition}
Let $u \in\Lip(\R^N)$, there exists a unique function $\xi$ in ${\rm L}^\infty_{\mu_\infty}$ such that
 \begin{equation*}
   \left.
    \begin{array}{l c l}
    &(u_n)\in {\rm Lip}(\R^N), &\ \mbox{ equiLipschitz}\\
    & u_n\rightarrow u, &\ \mbox{uniformly on }\R^N\\
    \end{array}
   \right\}
  \Rightarrow P_{\mu_\infty}(\cdot, \nabla u_n(\cdot))\ws \xi.
 \end{equation*}
\end{proposition}

\begin{definition}\label{tangrad}
The function $\xi$ appearing in the last proposition is called tangential gradient of $u$
to $\mu_\infty$ and is denoted by $\nabla_{\mu_\infty} u$.
\end{definition}

\begin{proposition}[Integration by parts formula]\label{intPart}
Let $\Psi \in\Lip(\R^N)$ and $\theta \in {\rm L}^1_{\mu_\infty}(\R^N,\R^N)$
such that $-{\rm div}(\theta \mu_\infty)$ belongs to ${\mathcal M}_b(\R^N)$.
Then
$$\theta(x)\in T_{\mu_\infty}(x)\ \mu_\infty-{a.e.},\quad \mbox{and} \quad -<{\rm div}(\theta {\mu_\infty}),\Psi>=
\int \theta\cdot \nabla_{\mu_\infty} \Psi\ d\mu_\infty.$$
\end{proposition}

In the previous results, we have defined the tangential gradient of functions in $\Lip(\R^N)$.
As we are dealing with functions on $\Lip(\Omega)$,
we will also need the following
$$ u\in  \Lip(\R^N),\ \quad u=0\ \mu_\infty\mbox{-a.e. in }\Omega \Rightarrow
\nabla_{\mu_\infty} u=0\  \mu_\infty\mbox{-a.e. in $\Omega$}$$
so that the tangential gradient of any function $u$ in $\Lip(\Omega)$ is well defined
via the restriction of the tangential gradient of any of its Lipschitz extension to $\R^N$.

\begin{proof}[Proof of Theorem \ref{main3}]
Using the duality relation between
$(\MP_\infty)$ and $(\MP^*_\infty)$ and the optimality of $\sigma_\infty=\xi_\infty\mu_\infty$ and $u_{\infty}$
(see Theorem \ref{thm-ptoinfty} and Proposition \ref{min-sigma-inf}), we get:
\begin{equation}\label{extr}
\int_{\Omega} u_{\infty}(x)\ df_\infty(x)=\Lambda_{\infty}\int_{\overline \Omega} \vert \xi_\infty(x)\vert\ d\mu_\infty(x).
\end{equation}
By Proposition \ref{intPart}, as  $-{\rm div}(\sigma_{\infty})\in {\mathcal M}_b(\R^N)$ and $u_\infty$ is zero outside $\Omega$, we can make
an integration by parts an get:
$$\int_{\Omega} u_{\infty}(x)\ df_\infty(x)
=\langle-{\rm div}(\xi_\infty\mu_\infty),u_{\infty}\rangle_{{\mathcal M}_b(\R^N), {\mathcal C}_c(\R^N)}
=\int_{\overline\Omega} \nabla_{\mu_\infty}u_{\infty}\cdot \xi_\infty\ d\mu_\infty.$$
Using (\ref{extr}), we get:
\begin{equation}\label{extr*}
\int_{\overline \Omega}  (\nabla_{\mu_\infty}u_{\infty}\cdot \xi_\infty)-\Lambda_{\infty}\vert \xi_\infty\vert\ d\mu_\infty=0.
\end{equation}
The constraint $\vert \nabla u_{\infty}\vert\le \Lambda_{\infty} \mbox{ a.e. in }\Omega$ is reformulated using the definitions
of $T_{\mu_\infty}$ and $\nabla_{\mu_\infty}$ as a constraint on $\nabla_{\mu_\infty} u_\infty$ by saying (see \cite{J}, Lemma 4.13 and proof of Theorem 5.1):
$$\exists \zeta\in {\rm L}_{\mu_\infty}^\infty(\R^N,\R^N)\mbox{ such that } \left\{
\begin{array}{l}
\displaystyle{ \zeta(x)\in T_{\mu_\infty}(x)^\perp,\quad \mu_\infty\mbox{-a.e.}x\in{\overline\Omega}}\\
\displaystyle{ \vert \nabla_{\mu_\infty} u_\infty(x) +\zeta(x) \vert \le \Lambda_{\infty},\quad  \mu_\infty\mbox{-a.e.}x\in{\overline\Omega}.}\\
\end{array}
\right.
$$
As $\xi_\infty(x)\in T_{\mu_\infty}(x)$ $\mu_\infty$-a.e, we have:
$$\nabla_{\mu_\infty} u_\infty(x)\cdot \xi_\infty(x)=(\nabla_{\mu_\infty} u_\infty(x)+ \zeta(x))\cdot \xi_\infty(x)\le \Lambda_\infty \vert \xi_\infty(x) \vert\ \mu_\infty\mbox{-a.e.}x\in{\overline\Omega}.$$
Combining this with (\ref{extr*}), we  obtain $\nabla_{\mu_\infty}u_{\infty}(x)\cdot \xi_\infty(x)= \Lambda_{\infty}\vert \xi_\infty(x) \vert$
$\mu_\infty-$almost everywhere and consequently:
$$\vert  \nabla_{\mu_\infty}u_{\infty}\vert =\Lambda_{\infty},\quad \xi_\infty={ \nabla_{\mu_\infty}u_{\infty}\over \Lambda_\infty}\quad \mu_\infty-\mbox{ a.e. in }{\overline \Omega}.$$
The second equality in (\ref{syst}) then follows from
$\sigma_\infty=\Lambda_\infty^{-1} \nabla_{\mu_\infty}u_{\infty}\,.\,\mu_\infty$.
\end{proof}

\CAAnewsection{5. SOME PROPERTIES OF THE LIMITS}\label{applicsect}

In this section we will use the optimal transport problem
to investigate more properties of $u_\infty$ and $f_\infty$ and to
give an alternative way of identifying $\Lambda_\infty$ which we hope
will be useful in the future.

We shall denote by $d_\Omega(x)$ the distance of a point $x$ of $\Omega$
from $\partial \Omega$ and we recall the notation
$$ R_1= \sup \{r | \ \exists \ x_0 \ \mbox{s.t.} \ B(x_0,r) \subset
\Omega\}.$$

The main theorem is the following:

\begin{theorem}\label{identifie} The limits  $u_\infty$, $f_\infty$ and
  $\Lambda_\infty$ satisfies the following:
\begin{enumerate}
\item $f_\infty$ maximizes $\Wa (\cdot, {\mathcal P}(\partial
  \Omega))$ in $\Pb(\Omega)$,
\item $\Lambda_\infty=\frac{1}{R_1}$,
\item $ spt (f_\infty) \subset argmax\ u_\infty \subset argmax\ d_\Omega.$
\end{enumerate}
\end{theorem}

\begin{proof}[Proof of Theorem \ref{identifie}]  By propositions
\ref{propriegamma} and \ref{maximize}
the couple
$(f_\infty,u_\infty)$ minimizes $G_\infty$ or, which is equivalent, maximizes
\begin{eqnarray*}
 \max \{\langle g, v \rangle \ | \ \int_\Omega d|g| \leq
1, \ v \in \W ^{1,\infty}_0 (\Omega), \ \|\nabla v\|_{\infty} \leq \Lambda_\infty \}\\
= \max_{ g \in \Pb(\Omega)}\max  \{\langle g, v \rangle \ | \ v \in \W
^{1,\infty}_0 (\Omega),
\ \|\nabla v\|_{\infty} \leq \Lambda_\infty \}\\
= \max_{ g \in \Pb(\Omega)} \Lambda_\infty \Wa (g, \Pb(\partial \Omega)).
\end{eqnarray*}
We now remark that $ \max_{ g \in \Pb(\Omega)} \Wa (g, \Pb(\partial
\Omega))=R_1$ and that the maximal value is achieved exactly by the
probability measures concentrated on the set $\{x \in \Omega \ | \
d_\Omega (x)=R_1\}= argmax\  d_\Omega $. 
Then $\Wa(f_\infty, \Pb (\partial \Omega))=R_1$ and 
$f_\infty$ is concentrated on the set $argmax\  d_\Omega$.
Then from $1=\Lambda_\infty \Wa (f_\infty, \Pb(\partial
\Omega))=\Lambda_\infty R_1 $  it follows $\Lambda_\infty=\frac{1}{R_1}$.

Let us now prove $argmax\ u_\infty\subset argmax \ d_{\om}.$\\
For $x\in \Omega$, let $y\in\partial \Omega$ be a projection of $x$ on $\Dom$, we have:
$$u_\infty(x)=u_\infty(x)-u_\infty(y)\le \vert\vert \nabla u_\infty\vert\vert_\infty
\vert x-y\vert={1\over R_1}d_\Omega(x).$$
Now, if $x$ is in  $argmax\ u_\infty$, $u_\infty(x)=1$ and using the inequality above we
get $1\le  {1\over R_1}d_\Omega(x)$
which implies $d_\om(x)=R_1$.

Finally, let us show that $spt\ f_\infty\subset argmax\ u_\infty$.\\
Assume $x$ is a point out of $argmax\ u_\infty$.
Then it exists a ball $B(x,r)$ centered at $x$ of radius $r$ on which $u_\infty<1-\alpha$
with $\alpha>0$. As $u_p\to u_\infty$
uniformly, for $p$ big enough we have $u_p<1-{\alpha\over 2}$ on $B(x,r)$. This statement
implies:
\begin{eqnarray*}
\int_{B(x,r)}\ df_\infty (y)&\le& \liminf_{p\to +\infty} \int_{B(x,r)}f_p(y)\ dy\\
=\liminf_{p\to +\infty}  \int_{B(x,r)} u_p(y)^{p-1} dy
&\le& \liminf_{p\to +\infty} (1-\alpha/2)^{p-1}\omega_N r^N=0.
\end{eqnarray*}
Consequently $x\not\in spt\ f_\infty$.

\end{proof}

\begin{remark}
Examples are given in \cite{JuuLinMan} to illustrate that $u_\infty$ may differ
from $d_\Omega$, but it is still an open question whether one has
$argmax u_\infty = argmax\ d_\Omega$. In this respect, a close understanding
on the transport problem $(\MP_\infty)$ may yield that $spt(f_\infty) = argmax\ d_\Omega$
and thus answer this question.
\end{remark}

Next step would be to investigate some PDE properties of $u_\infty$
with the aim of understanding in which region is satisfied each part
of the equation (\ref{inftyeigen}). We can give some partial results on
that.

\begin{definition} For each $x \in \Omega$ we define its projection on $\partial \Omega$ as
$$ p_{\partial \Omega} (x)= \{z \in \partial \Omega \ | \ |x-z|=d_\Omega (x)\}.$$
\end{definition}
The transport set $T$ is given by
\begin{equation}
T=\{[x,y]\ | \ x \in spt (f_\infty) \ \mbox{and}\ y \in p_\Omega (x)\}.
\end{equation}

The transport set plays a crucial role in the theory of optimal transportation because
 it is the set on which the transport takes actually place. It should also
play a role in dividing the open set $\Omega$ in regions in which $u_\infty$  
satisfies different equations. The next proposition below goes in this direction.

\begin{proposition} The function $u_\infty$ is differentiable in $T\setminus
(spt(f_\infty)\cup \partial \Omega)$ moreover it satisfies 
$-\Delta_{\infty}u_{\infty}\le 0$ in the viscosity sense on $T\setminus
(spt(f_\infty)\cup \partial \Omega)$.
\end{proposition}

\begin{proof}
Let $x_0\in T\backslash spt(f_\infty)$. There exists $(y_1,y_2)\in spt(f_\infty) \times
\partial \Omega \subset argmax\ (u_\infty) \times \partial \Omega$ such that $x_0\in ]y_1,y_2[$.
The closure of the segment $]y_1,y_2[$ is called a
transport ray and for each $z\in ]y_1,y_2[$, $u_\infty$ satisfies
$$ u_\infty (z)=\Lambda_\infty |z-y_2|=u_\infty (y_1)-\Lambda_\infty |z-y_1|.$$
It follows by a classical argument (see for example Proposition 4.2 of \cite{Amb}) that $u_\infty$ is differentiable on this 
segment and that $\vert \nabla u_{\infty}(z)\vert=\Lambda_{\infty}$ for all
$z\in]y_1,y_2[$. As $x_0\not \in argmax\ u_\infty$ one get
\begin{equation}\label{eq1}
\Lambda_{\infty}u(x_0)<\vert \nabla u(x_0)\vert=\Lambda_{\infty}.
\end{equation}
By \cite{JuuLinMan}, $u_\infty$ is a viscosity sub-solution of
$$\min\{ {\vert \nabla u(x)\vert\over \vert u(x)\vert}-\Lambda_\infty , -\Delta_\infty
u\}=0,$$
i.e. $ \forall x \in \Omega$ and for all smooth $\varphi $ such that 
$\varphi \geq u_\infty$ in $\Omega$ and $\varphi(x)=u_\infty (x)$ one has 
$$\min\{ {\vert \nabla \varphi(x)\vert\over \vert \varphi(x)\vert}-\Lambda_\infty , 
-\Delta_\infty \varphi(x)\}\leq 0.$$
The differentiability of $u_\infty$ at $x_0$ together with (\ref{eq1})
implies that for every $\varphi$ as above
$$\min\{ {\vert \nabla \varphi(x_0)\vert\over \vert \varphi(x_0)\vert}-\Lambda_\infty , 
-\Delta_\infty \varphi(x_0)\}=\min\{ {\vert \nabla u_\infty (x_0)\vert\over \vert 
u_\infty(x_0)\vert}-\Lambda_\infty , 
-\Delta_\infty \varphi(x_0)\}\leq 0,$$
and then 
  $-\Delta_\infty \varphi(x_0)\leq 0$ which is, by definition,  
$-\Delta_\infty u_\infty(x_0)\le 0$ in the viscosity sense.
\end{proof}


\CAAnewsection{ACKNOWLEDGMENTS}

The research of the second author is supported by  the project 
{\it``Metodi variazionali nella teoria
  del trasporto ottimo di massa e nella teoria geometrica della
  misura''} of the program PRIN 2006 of the Italian Ministry of the
University and by the {\it ``Fondi di ricerca di ateneo''}  of the  University of Pisa. 
Part of this paper was written while the third author was Post-doc at the Scuola Normale Superiore in Pisa. All the authors gratefully acknowledge the hospitality of the universities of
Pisa and Toulon.


\end{document}